\documentclass[11pt]{amsart}
\usepackage{amsthm, amssymb, latexsym, enumerate}
\usepackage[T1]{fontenc}
\usepackage[margin=1.1in]{geometry}
\usepackage{bbm}
\usepackage{color}
\usepackage{hyperref}
\hypersetup{
    bookmarks=true,         
    unicode=false,          
    pdftoolbar=true,       
    pdfmenubar=true,       
    pdffitwindow=false,    
    pdfstartview={FitH},      
    colorlinks=true,       % false: boxed links; true: colored links
   linkcolor=black,          % color of internal links
    citecolor=black,        % color of links to bibliography
    filecolor=black,      % color of file links
    urlcolor=black}           % color of external links

\def\ep{\varepsilon}

\def\wt{\widetilde}

\def\sk{\vskip 10pt}

\def\D{\mathbb D}
\def\C{\mathbb C}
\def\supp{{\text{\rm supp}\,}}

\def\noqed{\renewcommand{\qedsymbol}{}}

\newtheorem{thm}{Theorem}[section]
\newtheorem{lem}[thm]{Lemma}
\newtheorem{prop}[thm]{Proposition}
\newtheorem{cor}[thm]{Corollary}
\newtheorem{rem}[thm]{Remark}
\title{Sharp $C^{1,\bar1}$ estimates in K\"ahler quantization and non-pluripolar Radon measures}
\author{Zbigniew B\l ocki and Tam\'as Darvas}

\begin{document}

\begin{abstract}
Let $K_\varphi$ denote the weighted Bergman kernel associated to a plurisubharmonic function $\varphi$. We obtain upper bounds and positive lower bounds for the Bergman metric $i\partial \bar{\partial} \log K_\varphi$, expressed solely in terms of upper bounds and positive lower bounds of $i\partial \bar{\partial}\varphi$. Our approach applies in both local and compact K\"ahler settings. As an immediate application we obtain the optimal $C^{1,\alpha}$-convergence for the quantization of K\"ahler currents with bounded coefficients. We also show that any non-pluripolar Radon measure on a compact Kähler manifold admits a quantization.
\end{abstract}
\maketitle

\section{Main results}
\subsection{The $C^{1,\bar1}$ estimates for K\"ahler quantization.}

Let $X$ be a K\"ahler manifold and $(L,h) \to X$ an ample Hermitian line bundle with positive curvature $i\Theta(h) = -i\partial \bar \partial \log h=\omega>0$. We fix $(X,\omega)$ as our background K\"ahler structure.

In classical K\"ahler quantization one considers different Hermitian metrics $h'$ on $L$. Such a metric can be written as $h' = h e^{-\varphi}$, where the potential $\varphi \in \mathcal H_\omega$ belongs to the space of smooth K\"ahler potentials:
\[
\mathcal H_\omega := \{\varphi\in C^\infty(X) : \omega_\varphi := \omega + i \partial \bar \partial \varphi >0\}.
\]

Given $k \geq 1$, by  $\mathcal H^k$ we denote the space of Hermitian inner products on the vector space $H^0(X,L^k)$. The metric $h e^{-\varphi}$ induces a natural $L^2$ inner product via the \emph{Hilbert} map $H^k_{(\cdot)}:\mathcal H_\omega \to \mathcal H^k$:
\begin{equation}\label{eq: Hilb_map}
H^k_\varphi( s,s') : = \int_X h^k(s,s') e^{-k\varphi}
\omega^n, \ \ \ s,s' \in H^0(X,L^k).
\end{equation}
Let $\{e_j\}_{j =1,\ldots,N_k}$ be an orthonormal basis for $H^0(X,L^k)$ with respect to this Hilbert norm, where $N_k = \dim H^0(X,L^k)$. Accordingly, the $k$-th \emph{K\"ahler Bergman kernel/Bergman potential/Bergman measure} of $\varphi$ are defined as follows:
\begin{equation}\label{eq: K_k_def}
K^k_\varphi = \sum_{j =1}^{N_k} h^k(e_j,e_j) = \sup_{s \in H^0(X,L^k), \,H^k_\varphi(s,s) \leq 1} h^k(s,s),
\end{equation}
\begin{equation}\label{eq: P_k_def}
P^k_\varphi := \frac{1}{k}\log K^k_\varphi = \frac{1}{k} \log \bigg( \sum_{j =1}^{N_k} h^k(e_j,e_j) \bigg) = \sup_{s \in H^0(X,L^k), \, H^k_\varphi(s,s) \leq 1} \frac{1}{k} \log h^k(s,s).
\end{equation}

\begin{equation}\label{eq: M_k_def}
M^k_\varphi := \frac{(2\pi)^n}{k^n}   K^k_\varphi e^{-k\varphi}\omega^n.
\end{equation}

The important fact is that $P^k_\varphi \in \mathcal H_\omega$, and the $k$-th \emph{Bergman metric} $\omega_{P^k_\varphi}$ is the pullback of the Fubini--Study metric under an appropriate projective embedding of $X$ \cite[Chapter 7.1]{SzeBook}. Thus, $\omega_{P^k_\varphi}$ can be viewed as an \emph{algebraic K\"ahler metric}. 

A central theme in K\"ahler geometry, going back to a problem of Yau \cite{Yau}, has been to approximate the K\"ahler metric $\omega_\varphi$ using $\omega_{P^k_\varphi}$ as $k \to \infty$.
Due to early work of Tian \cite{Tian}, with later improvements by Bouche \cite{Bouche}, Ruan \cite{Ruan}, Catlin \cite{Catlin}, Zelditch \cite{Zelditch}, and Lu \cite{Lu}, we know that 
$$P^k_\varphi \to_{C^\infty} \varphi, \ \ \ M^k_\varphi\to_{C^\infty}  \omega_\varphi^n.$$
See \cite[Chapter 7]{SzeBook} for an exposition. \smallskip

But what can be said if $\omega_\varphi$ is a non-smooth/degenerate K\"ahler metric? When dealing with potentials $\varphi \in C^{m,\alpha}(X)$ coming from H\"older spaces, one can still define $P^k_\varphi$ in the same way, and it is natural to conjecture that $P^k_\varphi \to \varphi$ in $C^{m',\alpha'}$ for any $m' + \alpha' < m + \alpha$.

Since much of the K\"ahler quantization theory still relies on asymptotic expansions of smooth potentials, few results are available for H\"older-continuous potentials $\varphi$ (see \cite{BBS08,MaMarinescu} and references therein). In Tian's original result, he obtains control of $\omega_{P^k_\varphi}$ using the first five derivatives of $\varphi$ \cite{Tian}, and, to the best of our knowledge, subsequent studies have not substantially refined this aspect. Indeed, according to the discussion after \cite[Corollary 1.2]{DMN}, using presently availably tools, one needs that $\varphi$ is at least six times (!) differentiable to control the Bergman measure and potential.

The aim of our first main result is to specifically address this point. We show that upper bounds and positive lower bounds on $\omega_{P^k_\varphi}$ can be obtained only using upper bonds and positive lower bounds on $\omega_\varphi$. To state this sharp result, we consider the following space of degenerate K\"ahler potentials:
\[
\mathcal H_\omega^{1,\bar 1}:= \{\varphi \in C^{1,\bar 1}(X) \ \textup{ s.t. } \   \omega_\varphi := \omega + i\partial \bar \partial \varphi \geq 0 \ \textup{ as currents}\}.
\]
(By $C^{1,\bar 1}$ we mean functions whose Laplacians, or equivalently mixed complex derivatives, are globally bounded.)

\begin{thm}\label{thm: C^11_quant} Suppose that $\varphi \in \mathcal H^{1,\bar 1}_\omega$ satisfies 
$0 < a \omega \leq \omega_\varphi \leq A \omega.$ 
Then there exists a positive constant $C=C(X,\omega)$ such that
\begin{equation}\label{eq: Laplacian_est}
\frac{a^{n+1}}{C(1 + A^n)} \omega\leq \omega_{P^k_\varphi} \leq \frac{C(1 + A^{n+1})}{a^n} \omega \quad \textup{ for all } k \geq \frac{C}{a}.
\end{equation}
In particular, $P^k_\varphi \to_{C^{1,\alpha}} \varphi$ for any $\alpha \in (0,1)$.
\end{thm}

With \eqref{eq: Laplacian_est} in hand,  $P^k_\varphi \to_{C^{1,\alpha}} \varphi$ follows from the Sobolev inequality \cite{GT} and the $C^0$-convergence theorem of Demailly \cite{D4}, which states that if $\varphi \in C^0(X)$ and $\omega_\varphi \geq a \omega>0$ then $P^k_\varphi \to \varphi$ uniformly, as a consequence of the Ohsawa--Takegoshi extension theorem \cite{OT} (alternatively, one could also use \cite[Theorem 4.1]{DLR18}).

It is an interesting feature of our arguments that we always work with the Bergman
kernels on the diagonal. In the already mentioned works dealing with the smooth case they
always relied on asymptotics of the Bergman kernel off the diagonal (see also \cite{Fin26}, connecting off-diagonal expansion to Toeplitz operator composition and asymptotics). It would be interesting
to obtain similar estimates as in Theorem \ref{thm: C^11_quant} for higher-order derivatives,
also staying on the diagonal in the process.  

\subsection{The $C^{1,\bar 1}$ estimates for the local Bergman potential.}
The results in the compact K\"ahler case will be obtained by first proving corresponding theorems in the local case, which are of independent interest.

For a domain $\Omega \subset \mathbb C^n$ and plurisubharmonic (psh) function $u:\Omega \to [-\infty, \infty)$ we consider the (local) \emph{Bergman kernel} 
$$K_{\Omega,u}(z) =\sup\{|f(z)|^2\colon f\in\mathcal O(\Omega),\ \int_\Omega|f|^2e^{-u}d\lambda\leq 1\}.$$ 
This function is always smooth and $\log K_{\Omega,u}$ is plursubharmonic, naturally inducing a positive (1,1) current $i \partial \bar \partial \log K_{\Omega,u}$. When $\Omega$ and $u$ are bounded, this current becomes a smooth $(1,1)$ form. It is called the \emph{Bergman metric} with
weight $u$ having the following interpretation:
\begin{equation}\label{eq: Berg_metr_formula}
B^2_{\Omega,u}(z;v):=\left.\frac{\partial^2} {\partial\zeta\partial\bar\zeta}\right|_{\zeta=0} \log K_{\Omega,u}(z+\zeta v) =\frac{\wt K_{\Omega,u}(z;v)}{K_{\Omega,u}(z)},
\end{equation}
where $z\in\Omega$, $v\in\mathbb C^n$ and 
$$\wt K_{\Omega,u}(z;v):=\sup\{|f'(z)\cdot v|^2\colon f\in\mathcal O(\Omega),\ f(0)=0,\ \int_\Omega|f|^2e^{-u}d\lambda\leq 1\}.$$ 
Here we used the notation
$f'(z)\cdot v= \partial f (z)v=\left.\frac\partial {\partial\zeta}\right|_{\zeta=0}f(z+\zeta v).$ 
The extremal formula \eqref{eq: Berg_metr_formula} is well known, it is due to Bergman in the unweighted case and
his proof extends to the general case (see e.g. \cite[Theorem 2.7]{Bl}). \sk

    The following theorem provides various bounds on the kernels $K_{\Omega,u}$, $\wt K_{\Omega,u}$ and the meric $B^2_{\Omega,u}$.

\begin{thm}\label{mt}
Assume that $\Omega$ is a bounded pseudoconvex domain in $\mathbb C^n$, $B_r:=B(z,r)\Subset\Omega$ for some $z\in\Omega$ and $r>0$, $v\in\mathbb C^n$, and that $u$ is psh in $\Omega$. Then there exists $C=C(n,r,\operatorname{diam}(\Omega))$ such that
\begin{enumerate}[(i)]

\item\label{ub} If $u$ is $C^{1,\bar 1}(\overline{B_r})$ then 
\begin{equation}\label{ub2}
K_{\Omega,u}(z)\leq
Ce^{u(z)}\max(1,\sup_{B_r}\Delta u)^n
\  \textup{ and } \ 
\wt K_{\Omega,u}(z;v)\leq
Ce^{u(z)}\max(1,\sup_{B_r}\Delta u)^{n+1}|v|^2.
\end{equation}

\item\label{lb} \label{lb2} If $(\partial_j \partial_{\bar k}u)\geq a(\delta_{jk})$
in $B_r$ for some $a>0$, then 
\begin{equation}\label{lb1}
K_{\Omega,u}(z)\geq\frac 1{C}e^{u(z)}a^n \ \textup{ and } \ 
\wt K_{\Omega,u}(z;v)\geq\frac 1{C}e^{u(z)}a^{n+1}|v|^2.
\end{equation}
\item If all the above conditions hold, then
$$ \frac{a^{n+1}}{C\max(1,\sup_{B_r}\Delta u)^{n}} |v|^2 \leq B^2_{\Omega,u}(z;v) \leq  \frac{C\max(1,\sup_{B_r}\Delta u)^{n+1}}{a^n} |v|^2.$$
\end{enumerate}
\end{thm}

The estimates of (iii) are a simple consequence of \eqref{eq: Berg_metr_formula} and the estimates of (i) and (ii).
The condition $(\partial_j \partial_{\bar k}u)\geq a(\delta_{jk})$ for $a>0$ appears naturally in several complex variables, 
and it simply means that $\varphi$ is strongly psh.

The proof of the first estimate in \eqref{ub2} is essentially already contained in the argument of \cite[Theorem 2.1]{BB} where the authors establish weak convergence of certain Bergman measures - a result that we considerably extend in Theorem \ref{thm: weak_conv_Radon} below. For completeness, we reproduce the argument. 

The bounds in \eqref{lb1} will be obtained using techniques developed by the first author in \cite{B}. More precisely, in Theorem \ref{ot} we prove an optimal Ohsawa–Takegoshi extension theorem \cite{OT} for holomorphic functions with prescribed transversal derivatives, with strongly psh weights. The first estimate of \eqref{lb1} can also be obtained from \cite[Corollary 1.8]{GZ}, \cite[Theorem 4.1]{O}, \cite[Theorem 1.1]{NW}
or \cite[Theorem 1.2]{HTW}
(note however that Theorem \ref{ot} below improves the constant $\pi/\alpha$ in those results to the  optimal value $\pi(1-e^{-\alpha})/\alpha$).

\begin{cor} 
Let $u$ be a $C^{1,\bar 1}$ strongly psh function in
a bounded pseudoconvex domain $\Omega \subset \mathbb C^n$. Then by the above theorem and \eqref{eq: Berg_metr_formula} the Levi forms $i\partial \bar \partial u_k$ are locally uniformly bounded both from above and below
in $\Omega$, where
\[u_k:=\frac 1k\log K_{\Omega,ku}\]
is the Demailly approximation of $u$ \cite{D4}. Thus $u_k \to_{C^{1,\alpha}} u, \ \alpha \in (0,1)$.\qed
\end{cor}

Recall that if $u$ is merely psh then $u_k\to u$ pointwise and in $L^1_{loc}$. If
it is in addition continuous then $u_k\to u$ locally uniformly (see \cite{D4}, and also 
\cite{Bl}), and if $u$ is $C^\infty$ strongly psh then $u_k\to u$ in $C^\infty$ (see \cite{En}).  
Since the Levi forms \(i\partial \bar\partial u_k\) converge weakly to \(i\partial \bar\partial u\), this corollary shows the estimates of Theorem~\ref{mt}(iii) must be optimal. Indeed, \(i\partial \bar\partial u_k\) admit uniform upper bounds and positive lower bounds if and only if the same holds for \(i\partial \bar\partial u\). It is natural to ask whether these two optimal bounds can be decoupled. For instance, if \(u\) is strongly psh, does there necessarily exist a uniform positive lower bound for \(i\partial \bar\partial u_k\)? Likewise, does an upper bound on \(i\partial \bar\partial u\) imply corresponding uniform upper bounds for \(i\partial \bar\partial u_k\)? It is natural to expect that this upper bounded will be necessary to for establishing $C^{1,\alpha}$-convergence of the quantization of weak geodesics in the Mabuchi space of K\"ahler metrics; at present, this is only known in the toric setting \cite{SZ10} (c.f. \cite{PS06,  CS12, Br18, RZ10}.

\subsection{K\"ahler quantization of non-pluripolar Radon measures.} In K\"ahler quantization, the most general space of degenerate potentials one may consider is the space of $\omega$-psh potentials:
$$\textup{PSH}_\omega := \{\varphi: X \to [-\infty,\infty) \ \textup{ s.t. } \ \omega_\varphi \geq 0 \ \textup{ as currents}\}$$
Unfortunately, since $\varphi \in \textup{PSH}_\omega$ may have positive Lelong numbers, the inner product $H^k_\varphi(s,s')$ of \eqref{eq: Hilb_map} does not make sense for all $s,s' \in H^0(X,L^k)$.
For \eqref{eq: Hilb_map} to always make sense we need to require that $e^{-k\varphi}$ is (locally) integrable for all $k$. Skoda \cite{Skoda} 
(see also \cite{DemBook12}, p. 35) proved that if the Lelong number of a psh $u$ at $z$ is smaller than 1 then $e^{-2u}$ is integrable near $z$.
This means that our requirement is equivalent to the fact that Lelong numbers of $\varphi$ vanish at every point.   

A natural subclass of $\textup{PSH}_\omega$ on which one can hope to carry out Kähler quantization is the space $\mathcal E_\omega \subset \textup{PSH}_\omega$ of full mass potentials, introduced by Guedj and Zeriahi in \cite{GZ07}. Recall that a potential $\varphi \in \textup{PSH}_\omega$ has full mass if its non-pluripolar complex Monge--Amp\`ere measure satisfies
$$\int_X \omega^n_\varphi = \int_X \omega^n.$$
First, since the potentials $\varphi \in \mathcal E_\omega$ have zero Lelong numbers, all the concepts $K^k_\varphi, P^k_\varphi, M_\varphi^j$ from \eqref{eq: K_k_def},\eqref{eq: P_k_def},\eqref{eq: M_k_def} make sense allowing for quantization in this context. Second, the most important feature of the class $\mathcal E_\omega$ is that this set of potentials is the natural space of solutions for Calabi--Yau problems with non-pluripolar Radon measure on the right hand side. Namely if $\nu$ is a non-pluripolar Radon measure on $X$ with $\int_X \nu = \int_X \omega^n $ then due to \cite{GZ07, Di09} there exists a unique $\varphi_\nu \in \mathcal E$ with $\int_X \varphi_\nu \omega^n =0$ and
$$\omega_{\varphi_\nu}^n = \nu.$$
Consequently, one can introduce the smooth quantum measures on level $k$ accordingly:
$$M_\nu^k := M_{\varphi_\nu}^k$$
It is natural to hope that these quantize the measure $\nu$ in the large $k$-limit. We confirm this in our final main result, which can be viewed as a vast generalization of the classical theorem of Bouche–Catlin–Lu–Tian–Ruan–Zelditch \cite{Bouche,Ruan,Catlin,Zelditch,Lu}:

\begin{thm}\label{thm: weak_conv_Radon}
If $\nu$ is a non-pluripolar Radon measure on $X$ with $\int_X \nu = \int_X \omega^n$, then $M^k_\nu  \rightharpoonup  \nu$ weakly, as $k \to \infty$.
\end{thm}

As far as we are aware, the most general result available of this kind is that of Berman–Witt Nyström \cite{BW08} (c.f. \cite{BBWN11}), which covers only the case where $\varphi_\nu$ is continuous. Our first observation is that - using additional ideas from \cite{BFM} - the argument of \cite{BW08} applies also when $\varphi_\nu$ is merely bounded (or of finite energy). Of course, a typical element of $\mathcal E_\omega$ is neither bounded nor of finite energy \cite[Example 2.14]{GZ07}. This is precisely the main difficulty, which we overcome by combining the optimal estimates of Theorem \ref{ot} with the full mass version of Berndtsson’s quantum comparison principle \cite{Bern15}, thereby establishing Theorem \ref{thm: weak_conv_Radon}.
\smallskip

If one is comfortable working with partial Bergman kernels and partial Bergman measures (as in \cite{CMM17, DX24}), one can define $M_\varphi^k$ for any $\varphi\in \mathrm{PSH}_\omega$. Due to \cite[Theorem 1.1]{DX24}, a necessary requirement for $M_\varphi^k$ to converge weakly to $\omega_\varphi^n$ is that $\varphi$ have $\mathcal I$-good singularity type (in the sense of \cite{Xi24}). It is an interesting question whether this is also sufficient, and we will return to this problem in a sequel.

\subsection{Acknowledgments.} We thank Robert Berman,  Bo Berndtsson, Dan Coman, Eleonora Di Nezza, Siarhei Finski, Yu-Chi Hou, L\'aszl\'o Lempert, Johannes Testorf, Mingchen Xia, and Kewei Zhang  for discussions related to the topic of the paper over the years. 
Most of this research was done during the first author's stay at the University of Maryland for the academic year 2024/25. He is
grateful for the hospitality and stimulating research atmosphere. 
The first author was partially supported by the National Science Centre (NCN) grant no. 2023/51/B/ST1/01312.
The second author was partially supported by a Simons Fellowship and NSF grant DMS-2405274. 

\section{Preliminaries}

\subsection{Comparison of local Bergman kernels.} We recall a  well-known localization 
result concerning the Bergman kernels of a domain $\Omega$ and of a ball contained within it. The proof uses the H\"ormander estimate in a standard way.
\begin{lem} \label{lem: local_berg_compare}
Let $\Omega \subset \mathbb C^n$ be a bounded pseudoconvex domain and $u$ psh in $\Omega$. 
Suppose that $z_0 \in \Omega$, $v \in \mathbb C^n$ and $D_r(z_0) \Subset \Omega$ , where $D_r(z_0)$ is either 
the ball or the polydisk centered at $z_0$ with radius $r$. There exists $C=C(n,r,\text{diam}(\Omega))$ 
such that 
\begin{equation}\label{eq: K_local_est}
  K_{\Omega,u}(z_0)\leq K_{D_r(z_0),u}(z_0)\leq CK_{\Omega,u}(z_0),
\end{equation}
\begin{equation}\label{eq: wtK_local_est}
  \wt K_{\Omega,u}(z_0;v)\leq \wt K_{D_r(z_0),u}(z_0;v) 
               \leq C\wt K_{\Omega,u}(z_0;v).
       \end{equation}
\end{lem}

\begin{proof} By $C$ we will denote possibly different constants depending on the
required quantities (we will keep this convention later on as well).
We will only prove the estimates of \eqref{eq: K_local_est}, as the estimates of 
\eqref{eq: wtK_local_est} are argued almost the same way. The first inequality is clear.
To prove the second one  
we may assume that $z_0=0$, $D_r=D_r(0)$ and consider $f \in \mathcal O(D_r)$ with 
$\int_{D_r} |f|^2 e^{-u}d\lambda \leq 1$.
Let $\rho \in C^\infty_c(D_r)$ be a cutoff function such that $\rho|_{D_{r/2}} = 1$ 
and that $|\nabla \rho| \leq 10/r$. 
Using H\"ormander's estimate, we can produce a smooth function $\gamma \in C^\infty(\Omega)$ solving 
$\bar \partial \gamma = f \bar \partial \rho $ satisfying the estimate:
$$\int_\Omega |\gamma|^2 e^{-u -2n\log |z|}(1+|z|^2)^{-2}d\lambda 
   \leq \int_\Omega |\nabla \rho|^2 |f|^2 e^{-u-2n\log |z|}d\lambda 
   \leq C\int_{D_r} |f|^2 e^{-u}d\lambda\leq C.$$
Finiteness of the integral on the left  implies that $\gamma(0)=0$, hence $F:=f \rho - \gamma \in \mathcal O(\Omega)$ satisfies 
$$F(0) = f(0) \ \textup{ and } \ \int_\Omega |F|^2 e^{-u}d\lambda \leq C.$$ 
\vskip -20pt
\end{proof}

\subsection{Comparison of the local and K\"ahler Bergman kernels.}  For this subsection let $(X,\omega)$ be a compact K\"ahler manifold with a positive Hermitian line bundle $(L,h)$, with curvature $\Theta(h) = -i \partial \bar \partial \log h = \omega$.
\smallskip

We start with providing the following formula for the K\"ahler Bergman metric, the K\"ahler analog of \eqref{eq: Berg_metr_formula}. The proof is  standard and carries over from the local case (\cite[Theorem 2.7]{Bl}). 
\begin{lem}\label{lem: Bergman_formula} Let $\varphi \in \mathcal E_\omega$, $p \in X$ and $v \in T_p X$. We have the following formula:
\begin{flalign}\label{eq: Blocki_notes_formula}
\omega_{P^k_\varphi}\big|_p(v,v) &= \frac{\wt K^k_\varphi(p;v)}{k {K^k_\varphi}(p)},
\end{flalign}
where 
$$\wt K^k_\varphi(p;v)=\sup_{s \in H^0(X,L^k), \ s(p)=0, \ H^k_\varphi(s, s) \leq 1} h^k(\partial s_p(v),\partial s_p(v))$$
and  $\partial s_p: T_p X \to L^k_p$ is the `Jacobian map' of $s$ at its critical point $p$.
\end{lem}
\begin{proof}Let $e_1, \ldots, e_{N_k}$ be an ONB of $H^0(X,L^k)$ with respect to $H^k_{\varphi}$, such that $e_j(p) =0$ for $j \geq 2$ and $\partial {(e_j)}_p(v)=0$ for $j \geq 3$. We consider a coordinate chart $U$ near $p$. In this chart we can think of $e_j|_U$ as holomorphic functions and $h|_U$ is a positive valued function. Using \eqref{eq: K_k_def} we arrive at 
\begin{flalign*}
{\omega_{P^k_\varphi}}|_U &= \frac{i}{k} \partial \bar \partial \bigg( \sum_{j=1}^{N_k} |e_j|^2\bigg)
\end{flalign*}
Restricting to the value $p$ and plugging the vector $v \in T_p X$:
\begin{flalign}\label{eq: Bergman_metric_e1}
\omega_{P^k_\varphi}|_p(v,v) = \frac{1}{k}\frac{|\frac{\partial  e_1}{ \partial v} (p)|^2}{|e_0(p)|^2}= \frac{1}{k}\frac{h^k((\partial  e_1)_p(v),(\partial  e_1)_p(v))}{{K^k_\varphi}(p)},
\end{flalign}
where in the last step, we multiplied both the numerator and denominator with $h^k$. This immediately gives that the right hand side dominates the left hand side of \eqref{eq: Blocki_notes_formula}.
For the reverse direction, let $s \in H^0(X,L^k)$ with $H^k_\varphi(s,s) \leq 1$ and $s(p) =0$. Using the above ONB, we can write $s = \sum_{j = 2}^{N_k} {s_j } e_j$, where $\sum_{j = 2}^{N_k} |s_j|^2 \leq 1$.
We get that
$$\frac{1}{k} h^k(\partial s_p(v), \partial s_p(v)) = |s_1|^2 \frac{1}{k} h^k(\partial (e_1)_p(v), \partial (e_1)_p(v)) \leq \frac{1}{k} h^k(\partial (e_1)_p(v), \partial (e_1)_p(v)).$$
After dividing this inequality with $h^k(e_0(p),e_0(p))$ and comparing with \eqref{eq: Bergman_metric_e1}, we get that the left hand side dominates the right hand side in \eqref{eq: Blocki_notes_formula}.
\end{proof}

The proof of Lemma \ref{lem: local_berg_compare} relied on the fact that the trivial line bundle is positively curved on domains in $\mathbb{C}^n$. On a general complex manifold this is no longer the case. Nevertheless, a variant of Lemma \ref{lem: local_berg_compare} can still be established for manifolds, and this will play an important role in the context of Kähler quantization:

\begin{lem}\label{lem: Berg_kern_compare_compact} 
Assume that $\varphi\in \mathcal E_\omega$ is such that $\omega_\varphi \geq a \omega$ for some $a>0$.
Fix $p \in X$ and let $P(p,3)\subset X$ be a coordinate polydisk of (multi)radius $3$ 
centered at $p$ such that $h|_{P(p,3)} = e^{-g}$ for some $g \in C^\infty(P(p,3))$.
Then there exists a positive constant $C$, depending on this coordinate polydisk and $\omega$,  
such that for any $k \geq C/a$ and $z\in P(p,1)$ we have the following 
estimates between the K\"ahler and local Bergman kernels at $z$:
\begin{equation}\label{eq: K_est_compact}
  K^k_{\varphi}(z)\leq K_{P(z,1),k\varphi+kg - \log ( \det (2\omega))}(z)e^{-k g(z)} 
     \leq \bigg(1 + \frac{C}{(ka)^{1/2}}\bigg)K^k_{\varphi}(z),
\end{equation}
\begin{equation}\label{eq: wtK_est_compact}
  \wt K^k_{\varphi}(z;v)\leq\wt K_{P(z,1),k\varphi+kg - \log ( \det (2\omega))}(z,v)e^{-k g(z)} 
       \leq \bigg(1 + \frac{C}{(ka)^{1/2}}\bigg)\wt K^k_{\varphi}(z;v),
\end{equation}
where $\det(2\omega)$ is the determinant of the Hermitian matrix $2 \partial_j \partial_{\bar k} g$
and $v \in T_zX$. 
\end{lem}

\begin{proof} We only show the inequalities of \eqref{eq: K_est_compact}, as the arguments  
for \eqref{eq: wtK_est_compact} are almost the same. Let $f \in H^0(X,L^k)$. Then, 
$f|_{P(z,1)} \in \mathcal O(P(z,1)) \simeq H^0(P(z,1),L^k)$ and we have the following inequality
  $$ \int_{P(z,1)} |f|^2 e^{-k(g + \varphi) + \log ( \det (2\omega))}d\lambda
     = \int_{P(z,1)} h^k(f,f) e^{-k\varphi}\omega^n 
      \leq  \int_{X} h^k(f,f) e^{-k\varphi}\omega^n.$$
This gives the first inequality in \eqref{eq: K_est_compact}.

For the other inequality, we consider $f \in \mathcal O(P(z,1)) \simeq H^0(P(z,1),L^k)$ 
satisfying the estimate
$\int_{P(z,1)} |f|^2 e^{-k(g + \varphi) + \log ( \det (2\omega))}d\lambda \leq 1$.
As in the local case, we find an appropriate cutoff function $\rho \in C^\infty_c(P(z,1))$ equal to $1$ on $P(z,1/2)$, and aim to find a smooth section $\gamma \in C^\infty(X,L^k)$ solving $\bar \partial \gamma = f \bar \partial \rho$ such that $\gamma(z)=0$ and $\int_X h^k(\gamma,\gamma) e^{-k\varphi} \omega^n$ is under control. For this, we choose $k$ big enough so that $k \omega_\varphi - i\partial\bar \partial (\log ( \det (2\omega)) + 2n \rho \log |\cdot - z|) > \frac{ka}{2} \omega$ on $P(z,1)$. This can be done for $k \geq C/a$.

We apply H\"ormander's theorem for the bundle $L^k = (L^k \otimes K_X^*) \otimes K_X$ \cite[Corollary 5.3]{DemBook12}, and after possibly increasing $C$, we obtain: 
\begin{flalign*}
\int_X h^k(\gamma,\gamma) e^{-k\varphi } \omega^n  &\leq \int_X h^k(\gamma,\gamma) e^{-k\varphi - \rho \log |\cdot - z|^{2n}} \omega^n  
\leq \frac{2}{ka} \int_X |\nabla^\omega \rho|_\omega^2 h^k(f,f) e^{-k\varphi - \rho \log |\cdot - z|^{2n}}  \omega^n\\
&\leq \frac{C}{ka} \int_{P(z,1)} h^k(f,f) e^{-k\varphi} \omega^n\leq \frac{C}{ka}.
\end{flalign*}
This implies that $F = \rho f - \gamma \in H^0(X,L^k)$, $F(z) = f(z)$ and 
$$\int_X h^k(F,F) e^{-k\varphi} \omega^n \leq  \bigg(1 + \bigg(\frac C{ka}\bigg)^{1/2}\bigg)^{2}.$$
Since $(C/ka)^{1/2}$ dominates $C/ka$ for $k$ big enough, the estimate of \eqref{eq: K_est_compact} follows. 
\vskip -15pt
\end{proof}

\subsection{Berndtsson's quantized comparison principle} 
We will need the full mass version of Berndtsson’s quantized comparison principle \cite[Theorem 1.1]{Bern15}. Since the original statement is formulated under slightly different assumptions, we also include a brief argument:

\begin{thm}\label{thm: Bernds_comp}
Let $\varphi,\psi \in \mathcal E_\omega$. Then
\begin{equation}\label{eq: Bernds_comp}
\int_{\{ \psi \leq \varphi\}} M^k_\varphi \leq \int_{\{\psi \leq \varphi\}} M^k_\psi.
\end{equation}
\end{thm}
\begin{proof}
Since $\psi \leq \max(\varphi,\psi)$ on $X$, and $\psi = \max(\varphi,\psi)$ on the set $\{\psi \geq \varphi\}$, 
we observe that
\begin{equation}\label{eq: important_est}
\mathbbm{1}_{\{\psi \geq \varphi\}} M^k_\psi  \leq \mathbbm{1}_{\{\psi \geq \varphi\}} M^k_{\max(\varphi,\psi)}.
\end{equation}
Integrating the above measures on the set $\{\psi > \varphi\}$, by equality of total masses we obtain that 
$$\int_{\{\psi \leq \varphi\}} M^k_\psi \geq \int_{\{\psi \leq  \varphi\}} M^k_{\max(\varphi,\psi)}.$$
Using symmetry, \eqref{eq: important_est} also gives,
$$\mathbbm{1}_{\{\psi \leq \varphi\}} M^k_\varphi \leq \mathbbm{1}_{\{\psi \leq \varphi\}} M^k_{\max(\varphi,\psi)}.$$
Integrating this inequality and combining it with the previous one yields \eqref{eq: Bernds_comp}.
\end{proof}

For the above argument to go through, the only requirement is that 
$H^k_\varphi(\cdot,\cdot)$ and $H^k_\psi(\cdot,\cdot)$ be well defined on $H^0(X,L^k)$. 
Thus, the assumption $\varphi,\psi \in \mathcal E_\omega$ can be further relaxed to 
$\omega$-psh functions whose Lelong numbers vanish.

\section{The $C^{1,\bar 1}$ estimates of Bergman potentials}

\subsection{The local case.}

\begin{proof}[Proof of Theorem \ref{mt}]
We prove the first estimate of \eqref{ub2}, extracting the argument from  \cite[Theorem 2.1]{BB}. Perhaps surprisingly, only subharmonicity of $u $ 
will be needed. For $f\in\mathcal O(\Omega)$ we have
\begin{flalign}\label{eq: BB_est}
     \int_\Omega|f|^2e^{-u }d\lambda
        &\geq\int_{B_r}|f|^2e^{-u}d\lambda
        =\int_0^r\int_{\partial B_\rho}|f|^2e^{-u }d\sigma\,d\rho \nonumber \\
       &\geq\int_0^r|\partial B_\rho|\exp\left(\frac 1{|\partial B_\rho|}
          \int_{\partial B_\rho}(\log|f|^2-u )d\sigma\right)d\rho\\ 
       &\geq|f(z)|^2\int_0^r|\partial B_\rho|\exp\left(-\frac 1{|\partial B_\rho|}
          \int_{\partial B_\rho}u  d\sigma\right)d\rho, \nonumber
\end{flalign}  
where the second inequality is a consequence of the Jensen inequality and
the third one of subharmonicity of $\log|f|^2$. It follows that
  $$K_{\Omega,u }(z)\leq K_{B_r,u }(z)\leq
    \left(\int_0^r|\partial B_\rho|\exp\left(-\frac 1{|\partial B_\rho|}
          \int_{\partial B_\rho}u  d\sigma\right)d\rho\right)^{-1}.$$
If $\Delta u \leq 4nA$ for some $A \geq 1$, then $u -A|\cdot-z|^2$ is superharmonic, so
\begin{equation}\label{lap-est}
  \frac 1{|\partial B_\rho|}
          \int_{\partial B_\rho}u  d\sigma\leq u(z)+A\rho^2.
\end{equation}          
We get 
  $$K_{\Omega,u }(z)\leq e^{u (z)}\left(
         2n\omega_{2n}\int_0^r\rho^{2n-1}e^{-A\rho^2}d\rho\right)^{-1}
         =e^{u (z)}A^n\left(
         n\omega_{2n}\int_0^{Ar^2}s^{n-1}e^{-s}ds\right)^{-1},$$
where $\omega_k$ denotes the volume of the unit ball in $\mathbb R^k$.
This shows the first estimate of \eqref{ub2}.\smallskip

To prove the second estimate of \eqref{ub2} we need the
following lemma:
\noqed
\end{proof}

\begin{lem}\label{lem: log_est} Suppose that $f \in \mathcal O(B_r)$, where $B_r=B(0,r)$, with 
$f(0)=0$ and $\partial f (0) \neq 0$. Then for some dimensional constant 
$C_n>0$ we have 
  $$\int_{\partial B_\rho}  \log {|f|^2} d \sigma_\rho -  \log \rho^2 
     \geq \sup_{|v|=1}\log \big( C_n |\partial f(0)\cdot v|^2\big), \ \ \rho \in (0,r).$$
\end{lem}

Here $d \sigma_\rho$ is the standard surface area probability measure on $\partial B(0,\rho)$. In case $n=1$, this result is an immediate consequence of $\log |f(z)|/|z|$ being subharmonic on the disk $P(0,r)$, and $C_1=1$ in this case.

\begin{proof} After an orthonormal change of variables we can assume that $\partial f(0) = \lambda d z_1$, where $\lambda = |\partial f(0)| =  \sup_{ ||v||=1}  |\partial f (0)(v)|$.
Fix $\rho \in (0,r)$. Let $v \in \partial B_\rho$ momentarily fixed. Using the observation before the proof 
for the one variable function $\xi \to \log |f(\xi v)|^2 / |\xi v|^2$ to conclude that 
$$\int_{\partial \D}  \log {\frac{|f( \xi v)|^2}{|\xi v|^2}} dV(\xi) \geq \log \bigg |\partial f(0) \frac{v}{|v|}\bigg|^2 = 2 \log |\lambda|  + 2 \log \frac{|v_1|}{|v|},$$
where $\D$ is the unit disk.
Now we integrate both sides with respect to $d \sigma_\rho(v)$, $v \in \partial B_\rho$, to obtain that
$$ \int_{\partial B_\rho} \int_{\partial \D}  \log {\frac{|f( \xi v)|^2}{|\xi v|^2}} dV(\xi) d \sigma_\rho(v) \geq  2 \log |\lambda|  + 2  \int_{\partial B_\rho}  \log \frac{|v_1|}{|v|} d \sigma_\rho(v).$$
Using the Fubini  theorem on the left hand side and a change or variables on the right hand side we arrive at 
$$ \int_{\partial B_\rho}  \log {|f(z)|^2} d \sigma_\rho -  \log \rho^2 \geq  2 \log |\lambda|  +  2 \int_{\partial B_1}  \log |v_1| d \sigma_1(v)=\log  \big ( C_n |\partial f (0)|^2\big) .$$
\end{proof}

\begin{proof}[Proof of Theorem \ref{mt}, contd] 
Take $f \in \mathcal O(\Omega)$ with $f(z)=0$. We can assume that  $z =0$ and $\|v\|=1$.
Using Lemma \ref{lem: log_est} in the inequalities of \eqref{eq: BB_est} we obtain that
\begin{flalign}
     \int_\Omega|f|^2e^{-u }d\lambda
        &\geq\int_{B_r}|f|^2e^{-u }d\lambda
        =\int_0^r\int_{\partial B_\rho}|f|^2e^{-u }d\sigma\,d\rho \nonumber \\
       &\geq\int_0^r |\partial B_\rho|\exp\left(\frac 1{|\partial B_\rho|}
          \int_{\partial B_\rho}(\log|f|^2-u )d\sigma\right)d\rho \nonumber \\ 
       &\geq C_n |\partial f(0)\cdot v|^2\int_0^r\rho^2|\partial B_\rho|\exp\left(-\frac 1{|\partial B_\rho|}
          \int_{\partial B_\rho}u  d\sigma\right)d\rho. \nonumber
\end{flalign}  
Now using \eqref{lap-est} similarly as before we get that
  $$\begin{aligned}
\wt K_{\Omega,u }(z;v)&\leq\frac 1{C_n}
    \left(\int_0^r\rho^2|\partial B_\rho|\exp
         \left(-\frac 1{|\partial B_\rho|}
          \int_{\partial B_\rho}u  d\sigma\right)d\rho\right)^{-1}\\
    &\leq\frac{e^{u (z)}A^{n+1}}{C_n}\left(
         n\omega_{2n}\int_0^{Ar^2}s^{n+1}e^{-s}ds\right)^{-1}.
     \end{aligned}$$
This gives the second estimate in \eqref{ub2}.

\eqref{lb} Lemma \ref{lem: local_berg_compare} implies that 
  $$K_{\Omega,u }(z)\geq\frac 1CK_{B_r,u }(z),
   \ \ \ \ 
   \wt K_{\Omega,u }(z;v)\geq\frac 1C\wt 
       K_{B_r,u }(z;v)$$
for some $C=C(n,r,\text{diam}(\Omega))$. 
After applying an affine transformation,  we can assume that $z=0$, $v=(0,\dots,0,1)$. 

The estimates of \eqref{lb1} will be a consequence of 
the following version of the Ohsawa-Takegoshi theorem: 
the first one will follow after using it $n$ times for $m=0$, and the second one will 
follow after using it once for $m=1$ and $n-1$ times for $m=0$ (see Corollary 
\ref{cor: opt_polydisk} below, stated for the unit polydisk). 
\noqed
\end{proof}

\begin{thm}\label{ot} 
Assume that $\Omega$ is pseudoconvex in $\C^{n-1}\times\D$ and
$u - a |z_n|^2$ is psh in $\Omega$ for some $a\geq 0$. Then for every 
$f\in\mathcal O(\Omega')$, where $\Omega'=\Omega\cap\{z_n=0\}$,
and $m=0,1,2,\dots$
one can find $F\in\mathcal O(\Omega)$ such that for $z'\in\Omega'$,
\begin{equation}\label{co1}
  \frac{\partial^jF}{\partial z_n^j}(z',0)=0,\ \ \ \ 
       j=0,1,\dots,m-1,
\end{equation}
\begin{equation}\label{co2}       
  \frac{\partial^mF}{\partial z_n^m}(z',0)=m!f(z')
\end{equation}         
and
  $$\int_{\Omega}|F|^2e^{-u }d\lambda
     \leq C_a\int_{\Omega'}|f|^2e^{-u }d\lambda,$$
where
\begin{equation}\label{const}
  C_a=\int_\D|\zeta|^{2m}e^{-a|\zeta|^2}d\lambda
     =\frac\pi{a^{m+1}}\int_0^a\rho^me^{-\rho}d\rho.
\end{equation}     
\end{thm}

\begin{rem}{\rm Since the supremum in 
  $$K^{(m)}_{\D,a|\zeta|^2}(0)=\sup\left\{\frac{|f^{(m)}(0)|^2}
        {\int_\D|f|^2e^{-a|\zeta|^2}d\lambda}\colon f\in\mathcal O(\D),\ 
         f(0)=f'(0)=\dots=f^{(m-1)}(0)=0\right\}$$
is attained for $f=\zeta^m$, we see that the constant in Theorem \ref{ot} 
is optimal for every $a\geq 0$ and $m=0,1,2,\dots$
} 
\end{rem}

\begin{cor}\label{cor: opt_polydisk} Assume that $u-a|z|^2$ is psh in the unit polydisk 
$\D^n$ for some $a\geq 0$. Then for $v=(1,0,\dots,0)$ we have:
\begin{flalign*}
K_{\D^n,u }(0)&\geq \frac{a^n}{\pi^n(1-e^{-a})^n} e^{u (0)}
      \geq \frac{a^n}{\pi^n} e^{u (0)},
   \\
   \wt K_{\D^n,u }(0;v)&\geq\frac{a^{n+1}}{\pi^{n}(1-e^{-a})^n}e^{u (0)}
       \geq\frac{a^{n+1}}{\pi^{n}}e^{u (0)}. 
\end{flalign*}
\end{cor}

\sk

Theorem \ref{ot} will be proved using methods developed in \cite{B}.
We will use the following $\bar\partial$-estimate 
from there:

\nopagebreak

\begin{thm}\label{est}
Let $\alpha\in L^2_{loc,(0,1)}(\Omega)$ be a $\bar\partial$-closed
form in a pseudoconvex domain $\Omega$ in $\mathbb C^n$. Assume
that $\varphi$ is psh in $\Omega$, $\psi\in
W^{1,2}_{loc}(\Omega)$ is locally bounded from above and they satisfy
$|\bar\partial\psi|^2_{i\partial\bar\partial\varphi}\leq H$, where
$H\in L^\infty_{loc}(\Omega)$ is such that $H<1$ in $\Omega$ and
$H\leq\delta<1$ on $\supp\alpha$. Then there exists
$U\in L^2_{loc}(\Omega)$ solving $\bar\partial U=\alpha$ and such
that 
\begin{equation}\label{dbar-est}
\int_\Omega|U|^2(1-H)e^{2\psi-\varphi}d\lambda
  \leq\frac{1+\sqrt\delta}{1-\sqrt\delta}\int_\Omega|\alpha|^2_{i\partial
    \bar\partial\varphi}e^{2\psi-\varphi}d\lambda.
\end{equation}
\end{thm}

\sk

\begin{proof}[Proof of Theorem \ref{ot}]
By approximation we may assume that $\Omega$ is bounded and smooth,
$\varphi$ is smooth and strongly psh 
and that $f$ is defined near $\overline{\Omega'}$. 

For sufficiently small $\ep>0$ we pick $\chi=\chi_\ep\in C^{0,1}([-\infty,0])$ so that $\chi(-\infty)=1$ and $\chi(t)=0$ for 
$t\geq t_\ep:=2\log\ep$.  Then
  $$\alpha=\bar\partial\big(f(z')z_n^m\chi(2\log|z_n|)\big)
      =\frac{f(z')z_n^m\chi'(2\log|z_n|)}{\bar z_n}d\bar z_n$$
is well defined in $\Omega$ and 
$\supp\alpha\subset\{|z_n\leq\ep\}=\{t\leq t_\ep\}$, where we use
the notation $t=2\log|z_n|$.       
We will apply Theorem \ref{est} with this $\alpha$ and weights      
  $$\begin{aligned}\varphi&=2(m+1)\log|z_n|+h(2\log|z_n|)+u-a|z_n|^2\\      
     \psi&=g(2\log|z_n|),\end{aligned}$$ 
where $h=h_\ep\in C^{1,1}((-\infty,0])$ with 
$h'\geq 0$, $h''\geq 0$
and $g=g_\ep\in C^{0,1}((-\infty,0])$ will be determined later.
We have
  $$|\bar\partial\psi|^2_{i\partial\bar\partial\varphi}
     \leq\frac{(g'(2 \log|z_n|))^2}{h''(2 \log|z_n|)}=:H$$
and  
\begin{equation}\label{alest}
  |\alpha|^2_{i\partial\bar\partial\wt\varphi}
      \leq\frac{|f(z')|^2|z_n|^{2m}(\chi'(2 \log|z_n|))^2}{h''(2 \log|z_n|)}.
\end{equation}      

On $\{t>t_\ep\}$ the functions $h,g$ will be defined in such a way
that
  $$e^{-u}=(1-H)e^{2\psi-\varphi}.$$
This is equivalent to
\begin{equation}\label{eqgh}
  \left(1-\frac{(g')^2}{h''}\right)e^{2g-h-(m+1)t+ae^t}=1.
\end{equation}  
One of the lessons from \cite{B} is that to handle \eqref{eqgh} 
one can use the substitution $g=\log h'$. Then \eqref{eqgh}
becomes 
  $$(e^{-h})''=e^{(m+1)t-ae^t}.$$
We choose solutions $h_0,g_0$ given by
\begin{equation}\label{h0}
  e^{-h_0(t)}=\int_t^0\int_x^0e^{(m+1)y-ae^y}dy\,dx,
  \ \ \ \ g_0=\log h_0'.
\end{equation}  
For later use, note that
\begin{equation}\label{prim}
  h_0'(t_\varepsilon)e^{-h_0(t_\varepsilon)} = -(e^{-h_0})'(t_\varepsilon)
     =\int_{t_\varepsilon}^0e^{(m+1)y-ae^y}dy 
     = \frac{1}{a^{m+1}} \int^a_{a \varepsilon^2} \rho^m e^{-\rho} d \rho.
\end{equation}    
On $\{t>t_\ep\}$ we define $h,g$ by $h_0,g_0$.

On $\{t\leq t_\ep\}$ they will be defined in such a way that 
$H=(g')^2/h''=\ep$ and $g=h+const$. These conditions actually determine
them:
  $$\begin{aligned}
     h(t):=\begin{cases}-\ep\log(A+t_\ep-t)+B,\ \ \ &t\leq t_\ep,\\
                     h_0(t),&t>t_\ep,\end{cases}\\
     g(t):=\begin{cases}-\ep\log(A+t_\ep-t)+\wt B,\ \ \ &t\leq t_\ep,\\
                     g_0(t),&t>t_\ep.\end{cases}
     \end{aligned}$$ 
Since $h\in C^{1,1}$, $g\in C^{0,1}$, the constants are determined by
\begin{equation}\label{str}
  h_0(t_\ep)=-\ep\log A+B,\ \ \ h_0'(t_\ep)=\frac\ep A,\ \ \ 
    g_0(t_\ep)=-\ep\log A+\wt B.
\end{equation}

Recall that $\alpha$ is supported inside $\{|z_n| \leq \varepsilon\}$. As a result, 
due to \eqref{alest} the $\limsup$ of the right-hand side of \eqref{dbar-est} 
(as $\ep\to 0$) is dominated by  
  $$\wt C\int_{\Omega'}|f|^2e^{-u}d\lambda,$$
where, the constant $\wt C$ is determined after a  change of variables $t = \log |z_n|^2$, and the computation of the following integral using polar coordinates: 
  $$\begin{aligned}
    \wt C&=\pi\limsup_{\ep\to 0}\frac{1+\sqrt\ep}{1-\sqrt\ep}
       \int_{ \{|z_n| \leq \varepsilon\}} \frac{(\chi'(\log |z_n|^2))^2}{h''(\log |z_n|^2) |z_n|^2}e^{2g((\log |z_n|^2))-h(\log |z_n|^2)+ae^{\log |z_n|^2}} dV(z_n)\\
       &=\pi\limsup_{\ep\to 0}\frac{1+\sqrt\ep}{1-\sqrt\ep}
       \int_{-\infty}^{t_\ep}\frac{(\chi')^2}{h''}e^{2g-h+ae^t}dt\\
     &=\pi\limsup_{\ep\to 0}\frac{e^{2\wt B-B}}\ep\int_{-\infty}^{t_\ep}
            (\chi')^2(A+t_\ep-t)^{2-\ep}dt
       \end{aligned}.$$
Now we pick a particular choice of $\chi$, so that $\chi(-\infty) =1$, $\chi(t)=0$ for $t \geq t_\varepsilon$, and 
$$\chi'(t)=-\frac{1}{\int_{-\infty}^{t_\varepsilon} (A+t_\ep-l)^{\ep-2} dl}(A+t_\ep-t)^{\ep-2}, \ \ t \leq t_\varepsilon.$$
For such $\chi$ we will get
  \begin{equation*}
     \begin{aligned}
    \wt C&=\pi\limsup_{\ep\to 0}
        \frac{e^{2\wt B-B}}\ep\left(\int_{-\infty}^{t_\ep}
            (A+t_\ep-t)^{\ep-2}dt\right)^{-1}\\
      &=\pi\limsup_{\ep\to 0}
          \frac{A^{1-\ep}e^{2\wt B-B}}\ep\\
      &=\pi\limsup_{\ep\to 0}e^{2g_0(t_\ep)-h_0(t_\ep)-\log h_0'(t_\ep)}\\
      &=\pi\limsup_{\ep\to 0} h'_0(t_\ep)e^{-h_0(t_\ep)}\\
      &=C_a
       \end{aligned}
    \end{equation*}   
by \eqref{str}, \eqref{h0} and \eqref{prim}, where $C_a$ is given by \eqref{const}.

Thus, if $U:=U_\ep$ is the solution of $\bar\partial U=\alpha$ given by Theorem \ref{est}, 
then $F=F_\ep=fz_n^m\chi-U \in \mathcal O(\Omega)$. Since near $\Omega'$ we have 
$(1-H)e^{2\psi-\varphi}\geq const|z_n|^{-2(m+1)}(-\log|z_n|)^{-\ep}$,
it follows that $U$ vanishes on $\Omega'$ to $m$-th order, so we get 
\eqref{co1} and \eqref{co2}. 

Since the $L^2$ integral of $F_\ep$ is bounded on compact subsets of $\Omega$, $F_\ep$ is locally 
uniformly bounded in $\Omega$, hence it has a subsequence converging 
to a holomorphic function in $\Omega$ satisfying the required estimate.
\end{proof} 

\subsection{The compact K\"ahler case.}
We now prove the analogous estimates for Kähler quantization. Throughout this subsection, 
let $(X,\omega)$ be a compact Kähler manifold, and let $(L,h)$ denote a positive Hermitian 
line bundle over $X$ whose curvature form satisfies $\Theta(h) = -i \partial \bar\partial \log h = \omega$.

\begin{thm}Suppose that $\varphi \in \mathcal E_\omega$ and $v \in T_z X, \ z \in X$. 
There exists a constant $C=C(X,\omega)$ such that the following hold\\
\noindent (i) If $\varphi \in C^{1,\bar 1}(X)$ then for all $k \geq 1$ we have
$$ {K^k_\varphi} \leq C e^{k\varphi} k^{n}(1 + \sup_X |\Delta \varphi|^{n}),$$
$$\wt {K}^k_\varphi(z;v) \leq C e^{k\varphi} k^{n+1}(1 + \sup_X |\Delta \varphi|^{n+1})\omega(v,v).$$
\noindent (ii) If $\omega_\varphi \geq a \omega>0$ then for all $k \geq \frac{C}{a}$ we have:
$${K^k_\varphi}  \geq  \frac 1C e^{k\varphi} a^n  k^n,$$

$${\wt K^k_\varphi}(z;v)  \geq  \frac 1C e^{k\varphi} a^{n+1}  k^{n+1} \omega(v,v).$$
\end{thm}

\begin{proof} Both the estimates of (i) and (ii) follow immediately from Theorem \ref{mt} 
and Lemma \ref{lem: Berg_kern_compare_compact}. Indeed, due to compactness of $X$, Lemma 
\ref{lem: Berg_kern_compare_compact}  is applicable for all $z \in X, v \in T_z X$.
Also note that the first inequalities in 
\eqref{eq: K_est_compact} and \eqref{eq: wtK_est_compact}, used in (i), do not require
the assumption $\omega_\varphi\geq a\omega$.
\end{proof}

Using \eqref{eq: Blocki_notes_formula} we now obtain Theorem \ref{thm: C^11_quant}.

\section{Quantization of the Monge--Amp\`ere energy}

Let us first recall some relevant concepts from \cite{Do01,Do05} related to the Monge--Amp\`ere energy and its quantum version. 
By $\mathcal H^k$ we denote the space of positive definite Hermitian forms on $H^0(X,L^k)$ and by $I_k:\mathcal H^k \to \mathbb R$ 
the quantum Monge--Amp\`ere energy:
$$I_k(G)= -\frac{(2\pi)^n}{ k^{n+1}} \big( \log \det G - \log \det H^k_0\big),$$
In contrast to \cite{Do01,Do05}, our expressions include an extra factor of $(2\pi)^n$. This discrepancy stems from our choice 
of normalization
$$\lim_{k \to \infty} \frac{N_k}{k^n} = \frac{1}{(2\pi)^n} \int_X \omega^n =: \frac{V}{(2\pi)^n},$$
where $N_k = \dim H^0(X,L^k)$.
By $I:\textup{PSH}_\omega \cap L^\infty \to \mathbb R$ we denote the Monge--Amp\`ere energy:
$$I(\varphi)=\frac{1}{n+1} \sum_{j=0}^n \int_X \varphi \omega_\varphi^j \wedge \omega^{n-j}.$$
This is a monotone functional in the sense that $\psi \leq \varphi$ implies $I(\psi) \leq I(\varphi)$, 
allowing to extend the definition of $I$ to $\textup{PSH}_\omega$:
$$I(\psi) = \inf_{\varphi \in \textup{PSH}_\omega \cap L^\infty, \psi \leq \varphi} I(\varphi).$$
Naturally, $I(\varphi)=\infty$ for some $\varphi$, and we denote by $\mathcal E^1_\omega$ -- 
the space of finite energy potentials --  the finite locus of $I$:
$$\mathcal E^1_\omega := \{\varphi \in \textup{PSH}_\omega \ \textup{ s.t. } \ I(\varphi) > -\infty\}.$$
It is a rather standard result that $\mathcal E^1_\omega \subset \mathcal E_\omega$ \cite{GZ07}.

It was argued in \cite[Theorem 3.5]{BFM} (generalizing a result of Donaldson from the smooth case, c.f. \cite{BB10}) that 
$I_k$ quantizes $I$ on $\mathcal E^1_\omega$ via the Hilbert map $\mathcal E^1_\omega \to \mathcal H^k$(recall \eqref{eq: Hilb_map}):

\begin{prop}[\cite{BFM}]\label{prop: I_k_conv} For any $\varphi\in \mathcal E^1_\omega$ we have
$\lim_{k \to \infty} I_k(H^k_\varphi)=I(\varphi).$
\end{prop}

We will also need the following perturbed version of this proposition, having its roots in \cite{BW08}:

\begin{prop}\label{prop: perturb_dp_conv} Let $\varphi\in \mathcal E^1_\omega$ and $f \in C^\infty(X)$.  
Then  $I_k(H^k_{\varphi + f}) \to I(P(\varphi+f))$.
\end{prop}
In the above statement $P(\chi)$ stands for the envelope of the qpsh function $\chi$:
\begin{flalign}\label{eq: P_def}
P(\chi):= \sup\{\varphi \in \textup{PSH}_\omega, \ \ \varphi \leq \chi\}.
\end{flalign}
\begin{proof} Take $u_l \in \mathcal H_\omega$ with $u_l \searrow u$ \cite{BK07}. We have 
$P(\varphi+f) \leq \varphi+f \leq \varphi_l + f$, hence  due to monotonicity of $H^k$ and $I_k$, we have 
\begin{equation}\label{eq: d_p_ineq}
 H^k_{P(\varphi + f)} \geq H^k_{\varphi+f} \geq  H^k_{\varphi_l+f}
\end{equation}
(as Hermitian forms) and
\begin{equation}\label{eq: I_k_ineq}
 I_k(H^k_{P(\varphi + f)}) \leq I_k(H^k_{\varphi+f}) \leq  I_k(H^k_{\varphi_l+f}).
\end{equation}
Letting $k\to\infty$ in \eqref{eq: I_k_ineq} and using \cite[Theorem A]{BB10} together with Propostion \ref{prop: I_k_conv} 
we get
$$I(P(\varphi+f)) \leq \liminf_{k\to\infty} I_k(H^k_{\varphi+f}) 
\leq \limsup_{k\to\infty} I_k(H^k_{\varphi+f}) \leq I(P(\varphi_l + f)).$$
Letting $l \to \infty$,  since $I(P(\varphi_l + f)) \searrow I(P(\varphi+f))$, we obtain 
the result.
\end{proof}

The first step in proving Theorem \ref{thm: weak_conv_Radon} is the case
of finite energy functions:

\begin{thm}\label{thm: E^1_Berg_Meas_conv} Suppose that $\varphi \in \mathcal E^1_\omega$. 
Then $M^k_\varphi \to \omega_\varphi^n$ weakly.
\end{thm}

The argument follows the framework set up in \cite{BW08}. There it is argued that that the map 
$\mathbb R \ni t \to I_k(H^k_{\varphi+ t f}) \in \mathbb R$ is concave for any $\varphi \in \mathcal H_\omega$ 
and $f \in C^\infty(X)$. Since the potentials of $\mathcal E^1_\omega$ have zero Lelong numbers, and they can be approximated decreasingly by elements of $\mathcal H_\omega$, we obtain that $t \to I_k(H^k_{\varphi+ t f})$ is in fact concave for any 
$\varphi \in \mathcal E_\omega^1$. Using elementary arguments, one can prove that the map 
$t \to I(P(\varphi +tf))$ is also concave, since so is $I$ \cite{GZ07}.

Regarding the derivatives of these maps, an elementary calculation and \cite{BB10} (for a survey see \cite[Proposition 4.32]{Da18}) yields that
\begin{equation}\label{eq: deriv_I_P}
\frac{d}{dt}\Big|_{t=0} I_k(H^k_{\varphi+ t f})= \int_X f M^k_\varphi, \ \ \ \ 
\frac{d}{dt}\Big|_{t=0} I(P(\varphi+tf))= \int_X f \omega_\varphi^n
\end{equation}

\begin{proof}[Proof of Theorem \ref{thm: E^1_Berg_Meas_conv}] 
Let $f \in C^\infty(X)$.
By Proposition \ref{prop: perturb_dp_conv} $I_k(H^k_{\varphi+ t f}) \to I(P(\varphi+tf))$ for any $t \in \mathbb R$. 
We now claim that $\frac{d}{dt}\big|_{t=0} I_k(H^k_{\varphi+ t f}) \to  \frac{d}{dt}\big|_{t=0} I(P(\varphi+tf)).$ 
Using \eqref{eq: deriv_I_P} this will finish the proof.

Using concavity, for any $\varepsilon >0$ we have 
$$\frac{I_k(H^k_{\varphi+\ep f}) - I_k(H^k_\varphi) }{\varepsilon} 
\leq \frac{d}{dt}\Big|_{t=0} I_k(H^k_{\varphi+ t f})  
\leq \frac{I_k(H^k_{\varphi-\ep f}) - I_k(H^k_\varphi) }{-\varepsilon}.$$
By Proposition \ref{prop: perturb_dp_conv} 
$$\begin{aligned}\frac{I(P(\varphi+\varepsilon f)) - I(\varphi) }{\varepsilon}
&\leq \liminf_{k\to\infty} \frac{d}{dt}\Big|_{t=0} I_k(H^k_{\varphi+ t f})\\  
&\leq \limsup_{k\to\infty} \frac{d}{dt}\Big|_{t=0} I_k(H^k_{\varphi+ t f})  
\leq  \frac{I(P(\varphi-\varepsilon f)) - I(\varphi) }{-\varepsilon}.
\end{aligned}$$
Letting $\varepsilon \to 0$ and using \eqref{eq: deriv_I_P} finishes the proof. 
\end{proof}

\section{Quantization of Radon measures}

We start with a result that is a consequence of Lemma \ref{lem: Berg_kern_compare_compact}
and Corollary \ref{cor: opt_polydisk}. It can be viewed as the quantum analogue of the inequality 
$\omega_\varphi^n \geq a^n \omega^n$, valid when $\omega_\varphi \geq a\,\omega$.

\begin{prop}\label{prop: BD_est} Suppose that $\varphi \in \mathcal E_\omega$ is such that  
$\omega_\varphi \geq a \omega$ for some $a>0$. Then for some 
$\varepsilon_k = \varepsilon_k(a,\omega) \searrow 0$ we have  
$$M^k_\varphi\geq (1 -\varepsilon_k) a^n \omega^n.$$
\end{prop}

\begin{proof} Let $p \in X$ and $\varepsilon> 0$. By rescaling the components of normal coordinates at $p$, 
there exists a coordinate polydisk $P(p,3)$ such that for some $\beta>0$ we have 
$\omega(p) = \left.\beta i\partial\bar\partial|z|^2\right|_p$ and
  $$(1-\varepsilon)\beta i\partial\bar\partial|z|^2
    \leq\omega
    \leq(1+\varepsilon)\beta i\partial\bar\partial|z|^2$$
in $P(p,3)$ (slightly abusing precision, we will think of the coordinate polydisk $P(p,3)$ 
as being a subset of both $\mathbb C^n$ and $X$). 
We can also assume that $\omega|_{P(p,3)} = i\partial \bar \partial g$ for some 
$g \in C^\infty(P(p,3))$. As a result, there exists $k_0$, depending on $\ep$, $a$, $\omega$ and our
coordinate polydisk $P(p,3)$, such that on $P(p,2)$ for $k\geq k_0$
  $$i \partial \bar \partial (g + \varphi  - \frac{1}{k}\log (\det(2\omega))
     \geq a(1-2\varepsilon)\beta i\partial\bar\partial|z|^2.$$
Now for any $q\in P(p,1)$ we use the estimate of  Corollary \ref{cor: opt_polydisk} on 
$P(q,1)\subset P(p,2)$ to obtain that:
$$K_{P(q,1),kg + k\varphi - \log ( \det (2\omega))} e^{-(kg + k\varphi - \log ( \det (2\omega)))} (q)
   \geq k^n \frac{a^n}{\pi^n} (1-2\varepsilon)^n \beta^n.$$ 
By Lemma \ref{lem: Berg_kern_compare_compact} there exists a constant $C$,
also depending on $\ep$, $a$, $\omega$ and $P(p,3)$, such that
\begin{flalign*}
\frac{(2\pi)^n}{k^n}K^k_\varphi e^{-k\varphi}  \det(\omega)(q) 
    &\geq \big(1+Ck^{-1/2}\big)^{-1}\,\frac{\pi^n}{k^n}
K_{P(q,1),kg + k\varphi - \log (\det(2\omega))} e^{-kg - k\varphi + \log ( \det(2\omega))}(q)\\
&\geq \big(1+Ck^{-1/2}\big)^{-1}\,(1-2\varepsilon)^n a^n \beta^n \\
&\geq  \big(1+Ck^{-1/2}\big)^{-1}\, a^n \frac{(1-2\varepsilon)^n}{(1+\varepsilon)^n}  \det(\omega)(q) , \ \ \ \ k \geq k_0.
\end{flalign*}
This implies that 
$$M^k_\varphi\geq \big(1+Ck^{-1/2}\big)^{-1}\,a^n \frac{(1-2\varepsilon)^n}{(1+\varepsilon)^n}  
\omega^n$$
on $P(p,1)$. By compactness this holds on $X$ for $k\geq k_0$ and $C$ depending on $a$, $\ep$ 
and $\omega$, hence 
  $$M^k_\varphi\geq\,a^n \frac{(1-2\varepsilon)^n}{(1+\varepsilon)^{n+1}}\omega^n,\ \ \ \ \ 
       k\geq k_0=k_0(a,\ep,\omega).$$
As $\varepsilon>0$ was arbitrary, the result follows.
\end{proof}

The next estimate is the quantum version of the inequality 
$\frac{1}{2^n} \omega_\varphi^n \leq \omega_{\varphi/2}^n,$ for $\varphi\in\mathcal E_\omega$ 
that follows from the multilinearity of the non-pluripolar complex Monge-Amp\`ere measure  \cite{GZ07}.
It can be deduced as a consequence of Finski's very general 
$L^{\infty}$ Ohsawa--Takegoshi theorem \cite[Theorem~1.10]{Fi21}. Here we 
give an alternative direct proof based on classical K\"ahler quantization:

\begin{prop}\label{prop: Fin_quant} Let  $\varphi \in \mathcal E_\omega$. Then for some 
$\varepsilon_k \searrow 0$ we have  
$$\frac{(1-\varepsilon_k)}{2^n} M^k_\varphi \leq M^{2k}_{\varphi/2}.$$ 
\end{prop}

\begin{proof} Recall that 
$$K^k_0 = \sup_{s \in H^0(X, L^k), \ \int_X h^k(s,s) \omega^n \leq 1}  h^k(s,s).$$
By the Catlin-Lu-Tian-Zelditch theorem, we known that $\frac{(2\pi)^n}{ k^n} K^k_0\rightrightarrows 1 $ on $X$. Hence there exists $D_k \searrow 1$ such that for every $p \in X$ there exists $s_p \in H^0(X, L^k)$ with 
\begin{equation}
\frac{1}{D_k} {k^n} \frac{1}{(2\pi)^n} \leq h^k(s_p,s_p)(p) \ \  \textup{ and }\ \ h^k(s_p,s_p)(x) \leq D_k k^n \frac{1}{(2\pi)^n}, \ x \in X.
\end{equation}

Let $\sigma \in H^0(X,L^k)$ with $\int_X h^k(\sigma, \sigma) e^{-k\varphi} \omega^n \leq 1$. Then,  $\chi_{\sigma, p} := \frac{(2\pi)^{n/2}}{D_k^{1/2} k^{n/2}}s_p \otimes \sigma \in H^0(X,L^{2k})$ satisfies
$$  \frac{1}{D_k^2}  h^{k} (\sigma, \sigma)(p) \leq h^{2k}(\chi_{\sigma, p},\chi_{\sigma, p})(p)  \ \ \textup{ and }\ \    \int_X h^{2k}(\chi_{\sigma, p},\chi_{\sigma, p}) e^{-k\varphi} \omega^n \leq 1.$$
By the definition of Bergman measures \eqref{eq: M_k_def} and the extremal definition of the Bergman kernel \eqref{eq: K_k_def}, the desired inequality $\frac{1}{2^n D_k^2} M^k_\varphi \leq M^{2k}_{\varphi/2}$ follows.
\end{proof}

The next technical result will allow to reduce the general case of Theorem \ref{thm: weak_conv_Radon} to the finite energy case already addressed in Theorem \ref{thm: E^1_Berg_Meas_conv}:

\begin{prop}Let $\varphi \in \mathcal E_\omega$. For any $\varepsilon >0$ there exists $c<0$ such that 
\begin{equation}\label{eq: B_k_est}
\limsup_{k \to \infty} \int_{\{\varphi \leq c \}} M^k_\varphi \leq   \varepsilon.
\end{equation}
\end{prop}

\begin{proof} We can assume that $\varphi \leq 0$. Fix $\alpha > 1$.
Recall that $P(\alpha \varphi) \in \mathcal E_\omega$ as follows from the argument of 
\cite[Proposition 2.15]{DLR18}. Specifically, this last result argues that $P(\alpha \varphi) \in \mathcal E^p_\omega$ 
if $\varphi \in \mathcal E^p_\omega$, but the same elementary proof applies for any class 
$\mathcal E_\omega^\chi$, with such classes exhuasting $\mathcal E_\omega$ (\cite[Proposition 2.2]{GZ07}).
We have 
$$\{\varphi \leq c\} \subset \Big\{\frac{1}{\alpha} P(\alpha \varphi) \leq \frac\varphi{2} + \frac{c}{2}\Big\}.$$
This follows from the fact that $\frac{1}{\alpha} P(\alpha \varphi) \leq \varphi \leq \varphi/2 + c/2$ on the set $\{\varphi \leq c\}$. As a result, we can use Proposition \ref{prop: Fin_quant} and Theorem \ref{thm: Bernds_comp} to obtain
\begin{flalign}\label{eq: M_k_half}
\frac{1-\ep_k}{2^n}\int_{\{\varphi \leq c \}} M^{k}_\varphi & 
\leq \int_{\{\varphi \leq c \}} M^{2k}_{\varphi/2} 
\leq \int_{\{\frac{1}{\alpha} P(\alpha \varphi) \leq \frac{\varphi}{2} + \frac{c}{2}\}}M^{2k}_{\varphi/2} \\
\nonumber &\leq \int_{\{\frac{1}{\alpha} P(\alpha \varphi) 
\leq \frac{\varphi}{2} + \frac{c}{2}\}} M^{2k}_{\frac{1}{\alpha}P(\alpha \varphi)}
\leq \int_{\{\frac{1}{\alpha} P(\alpha \varphi) \leq \frac{c}{2}\}} M^{2k}_{\frac{1}{\alpha}P(\alpha \varphi)}.
\end{flalign}
Trivially, we have $\omega_{\frac{1}{\alpha}P(\alpha \varphi)} \geq \big(1 - \frac{1}{\alpha}\big) \omega$. Proposition \ref{prop: BD_est} now implies that
$$M^{2k}_{\frac{1}{\alpha}P(\alpha \varphi)} \geq (1 - \wt\varepsilon_k) \Big(1 - \frac{1}{\alpha}\Big)^n\omega^n.$$
Using this, we can continue the estimates of \eqref{eq: M_k_half}:
\begin{flalign*}
\frac{1-\ep_k}{2^n}\int_{\{\varphi \leq c \}} M^{k}_\varphi 
&\leq \int_X M^{2k}_{\frac{1}{\alpha}P(\alpha \varphi)}- 
   \int_{\{\frac{1}{\alpha} P(\alpha \varphi) > \frac{c}{2}\}}M^{2k}_{\frac{1}{\alpha}P(\alpha \varphi)}\\
& \leq \frac{(2\pi)^nN_k}{k^n} - (1-\wt\varepsilon_k) \Big(1-\frac{1}{\alpha}\Big)^n\int_{\{\frac{1}{\alpha} P(\alpha \varphi) > \frac{c}{2}\}} \omega^n.
\end{flalign*}
There exists $c = c(\varphi,\alpha)<0$ such that  $\int_{\{\frac{1}{\alpha} P(\alpha \varphi) > \frac{c}{2}\}} \omega^n \geq (1 - \frac{1}{\alpha}) \int_X \omega$. Since $\frac{(2\pi)^nN_k}{k^n} \to \int_X \omega^n$, we obtain  
\begin{equation}
\limsup_{k \to \infty} \int_{\{\varphi \leq c \}} M^k_\varphi \leq   2^n\bigg(1 - \Big(1 - \frac{1}{\alpha}\Big)^{n+1} \bigg) \int_X \omega^n.
\end{equation}
We can now choose $\alpha>1$ so that the right-hand side does not exceed $\ep$.
\end{proof}

We can now finish the proof of Theorem \ref{thm: weak_conv_Radon}:

\begin{thm} 
Let $\varphi \in \mathcal E_\omega$. Then
$M^k_\varphi  \to  \omega_\varphi^n$ weakly, as $k \to \infty$.
\end{thm}

\begin{proof}
We can assume that $\varphi \leq 0$. Let $\varepsilon_m := 1/m $ and $c_m\to -\infty$ be given by the previous proposition. 
Then using \eqref{eq: important_est} we can write
\begin{flalign*}
M^k_\varphi&= \mathbbm{1}_{\{\varphi \leq c_m\}}M^k_\varphi+\mathbbm{1}_{\{\varphi > c_m\}}M^k_\varphi\\
& \leq \mathbbm{1}_{\{\varphi \leq c_m\}}M^k_\varphi+\mathbbm{1}_{\{\varphi> c_m\}}M^k_{\max(\varphi,c_m)}\\
&\leq \mathbbm{1}_{\{\varphi \leq c_m\}}M^k_\varphi +M^k_{\max(\varphi,c_m)}.
\end{flalign*}
As a result, by Theorem \ref{thm: E^1_Berg_Meas_conv} and  \eqref{eq: B_k_est}, we obtain that any weak limit $\mu$ of 
$M^k_\varphi$ is dominated by $\mu_{c_m} + \omega^n_{\max(\varphi,c_m)}$, where the total mass of $\mu_{c_m}$ on $X$ 
is less than $\varepsilon_m = 1/m$. Letting $m \to \infty$, since $\omega^n_{\max(\varphi,c_m)} \to \omega_\varphi^n$ weakly \cite{GZ07}, 
we conclude that $\mu \leq \omega_\varphi^n$. Due to equality of total masses, we must have $\mu = \omega_\varphi^n$, 
as desired.
\end{proof}

\vspace{0.3in}
\noindent {\sc Jagiellonian University,
Institute of Mathematics, Krak\'ow, Poland}\\
{\tt zbigniew.blocki@uj.edu.pl }\vspace{0.1in}

\noindent {\sc University of Maryland, Deparment of Mathematics, College Park, USA}\\
{\tt tdarvas@umd.edu}\vspace{0.1in}\\


\begin{thebibliography}{99}

\def\wtt{$\widetilde{\phantom{a}}$}

\bibitem{Berm09} R. Berman, {\sl Bergman kernels and equilibrium measures for line bundles over projective manifolds}, Amer. J. Math. 131 (2009), 1485--1524

\bibitem{BB} R. Berman, B. Berndtsson, {\sl Convexity of the $K$-energy on the space of K\"ahler metrics and uniqueness of extremal metrics}, J. Amer. Math. Soc. 30 (2017), 1165--1196

\bibitem{BBS08} R. Berman, B. Berndtsson, J. Sj\"ostrand, {\sl A direct approach to Bergman kernel asymptotics for positive line bundles}, Ark. Mat. 46 (2008), 197--217

\bibitem{BB10} R. Berman, S. Boucksom, {\sl Growth of balls of holomorphic sections and energy at equilibrium}, Invent. Math. 181 (2010), 337--394

\bibitem{BFM} R. Berman, G. Freixas i Montplet, {\sl An arithmetic Hilbert-Samuel theorem for singular hermitian line bundles and cusp forms}, Compos. Math. 150 (2014), 1703--1728

\bibitem{BW08} R. Berman, D. Witt Nystr\"om, {\sl Convergence of Bergman measures for high powers of a line bundle},  arXiv:0805.2846

\bibitem{BBWN11} R. J. Berman, S. Boucksom, and D. Witt Nystr\"om. Fekete points and convergence towards equilibrium measures on complex manifolds. Acta Mathematica 207.1 (2011), pp. 1-27.

\bibitem{Br18} B. Berndtsson, Probability measures associated to geodesics in the space of Kähler metrics. Algebraic and analytic microlocal analysis, 395--419, Springer Proc. Math. Stat., 269, Springer, Cham, 2018.


\bibitem{Bern15} B. Berndtsson, {\sl A comparison principle for Bergman kernels}, in {\sl Analysis Meets Geometry: A Tribute to Mikael Passare}, eds. M. Andersson, J. Boman, C. Kiselman, P. Kurasov, R. Sigurdsson, Trends in Mathematics, pp. 121--126, Springer, 2017 

\bibitem{B} Z. B\l ocki, {\sl Suita conjecture and the Ohsawa--Takegoshi extension theorem}, Invent. Math. 193 (2013), 149--158

\bibitem{Bl} Z. B\l ocki, {\sl Selected Topics in Several Complex Variables}, course given at the University of Maryland, Spring 2025,
available at 
{\tt http://gamma.im.uj.edu.pl/\wtt blocki}

\bibitem{BK07} Z. B\l ocki, S.Ko\l odziej, {\sl On regularization of plurisubharmonic functions on manifolds}, Proc. Amer. 
Math. Soc. 135 (2007), 2089--2093

\bibitem{Bouche} T. Bouche, {\sl Convergence de la m\'etrique de Fubini--Study d'un fibr\'e lin\'eaire positif}, Ann. Inst. Fourier 40 (1990), 117--130

\bibitem{Catlin} D. Catlin, {\sl The Bergman kernel and a theorem of Tian}. In {\sl Analysis and Geometry in Several Complex Variables}, 
Proc. 40th Taniguchi Symposium, eds. G. Komatsu, M. Kuranishi, pp. 1--23, Birkh\"auser, 1999

\bibitem{CS12} X.X. Chen, S. Sun, Space of Kähler metrics (V)—Kähler quantization. Metric and differential geometry, 19--41, Progr. Math., 297, Birkhäuser/Springer, Basel, 2012.

\bibitem{CMM17} D. Coman, X. Ma, G. Marinescu, Equidistribution for sequences of line bundles on normal Kähler spaces. Geom. Topol. 21 (2017), no. 2, 923--962.

\bibitem{Da15} T. Darvas, {\sl The Mabuchi geometry of finite energy classes}, Adv. Math. 285 (2015), 182--219

\bibitem{Da18} T. Darvas, {\sl Geometric pluripotential theory on Kähler manifolds}, in {\sl Advances in complex geometry}, eds. 
Y.A. Rubinstein, B. Shiffman, Contemp. Math. 735, pp. 1--104, Amer. Math. Soc., 2019

\bibitem{DLR18} T. Darvas, C.H. Lu, Y.A. Rubinstein, {\sl Quantization in geometric pluripotential theory},  
Comm. Pure Appl. Math. 73 (2020), 1100--1138

\bibitem{DX24} T. Darvas, M. Xia, The volume of pseudoeffective line bundles and partial equilibrium, Geom. and Topol. 28 (2024), no. 4, 1957-1993. 


\bibitem{D4} J.-P. Demailly, {\sl Regularization of closed positive currents and intersection theory}, J. Algebraic Geom. 
1 (1992), 361--409

\bibitem{DMN} T.-C. Dinh, X. Ma, V.-A. Nguyên,  {\sl On the asymptotic behavior of Bergman kernels for positive line 
bundles}, Pacific J. Math. 289 (2017), 71--89

\bibitem{DemBook09} J.-P. Demailly, {\sl Complex Analytic and Differential Geometry}, available at \newline
{\tt https://www-fourier.univ-grenoble-alpes.fr//\wtt demailly/}

\bibitem{DemBook12}  J.-P. Demailly, {\sl Analytic methods in algebraic geometry},
Higher Education Press, Surveys of Modern Mathematics, Vol. 1, 2010

\bibitem{Di09} S. Dinew, {\sl Uniqueness in $\mathcal E(X,\omega)$}, J. Funct. Anal. 256 (2009), 2113--2122

\bibitem{Do01} S. K. Donaldson, {\sl Scalar curvature and projective embeddings}, I, J. Diff. Geom. 59 (2001), 479--522

\bibitem{Do05} S. K. Donaldson, {\sl Scalar curvature and projective embeddings, II}, Quart. J. Math. 56 (2005), 345--356

\bibitem{En} M. Engliš, {\sl Weighted Bergman kernels and 
quantization}, Comm. Math. Phys. 227 (2002), 211--241

\bibitem{Fi21} S. Finski, {\sl Semiclassical Ohsawa-Takegoshi extension theorem and asymptotics of the orthogonal Bergman kernel}, J. Differential Geom. 128 (2024), no. 2, 639--721.
\bibitem{Fin26} S. Finski, 
{\sl Bernstein-Markov measures and Toeplitz theory},
arXiv:2506.01610.
\bibitem{GZ} Q. Guan, X.Y Zhou, {\sl Strong openness of multiplier ideal sheaves and optimal
$L^2$ extension}, Sci. China Math. 60 (2017), 967--976

\bibitem{GZ07} V. Guedj, A. Zeriahi, {\sl The weighted Monge-Amp`ere energy of quasiplurisubharmonic
functions}, J. Funct. Anal. 250 (2007), 442--482

\bibitem{GT} D. Gilbarg, N. Trudinger, {\sl Elliptic Partial Differential Equations of Second Order}, Classics in Math. 
Springer, 2001

\bibitem{HTW} Y. He, J. Testorf, X. Wang, {\sl Ross-Witt Nyström correspondence and Ohsawa-Takegoshi 
extension}, arXiv:2311.03840v5


\bibitem{Lu} Z. Lu, {\sl On the lower order terms of the asymptotic expansion of Tian--Yau--Zelditch}, 
Amer. J. Math. 122 (2000), 235--273.

\bibitem{MaMarinescu} X. Ma, G. Marinescu, {\sl Holomorphic Morse Inequalities and Bergman Kernels}, Progress in Math. 
254, Birkh\"auser, 2007

\bibitem{NW} T.T.H. Nguyen, X. Wang, {\sl On a remark by Ohsawa related to the Berndtsson-Lempert method for
$L^2$-holomorphic extension}, Ark. Mat. 60 (2022), 173--182

\bibitem{O} T. Ohsawa, {\sl On the extension of $L^2$ holomorphic functions VIII - a remark on a
theorem of Guan and Zhou}, Internat. J. Math. 28 (2017), art. no 1740005, 12 pp.

\bibitem{OT} T. Ohsawa, K. Takegoshi, {\sl On the extension of $L^2$ holomorphic functions}, Math. Z. 195 (1987), 197--204

\bibitem{SZ10} J. Song, S. Zelditch, Bergman metrics and geodesics in the space of Kähler metrics on toric varieties. Anal. PDE 3 (2010), no. 3, 295--358.

\bibitem{PS06} D.-H. Phong, J. Sturm, The Monge-Ampere operator and geodesics in the space of K\"ahler potentials, Invent. Math. 166 (2006), no. 1, 125–-149.

\bibitem{Ruan} W.-D. Ruan, {\sl Canonical coordinates and Bergman metrics}, Comm. Anal. Geom. 6 (1998), 589--631

\bibitem{RZ10} Y.A. Rubinstein, S. Zelditch, Bergman approximations of harmonic maps into the space of K\"ahler metrics on toric varieties.
J. Symplectic Geom. 8 (2010), no. 3, 239–265.

\bibitem{Skoda} H. Skoda, {\sl Sous-ensembles analytiques d’ordre fini ou infini dans $\mathbb C^n$}, 
Bull. Soc. Math. France 100 (1972), 353--408

\bibitem{SzeBook} G Székelyhidi, {\sl An introduction to extremal Kähler metrics}. Graduate Studies in Math. 152, 
Amer. Math. Soc., 2014 

\bibitem{Tian} G. Tian, {\sl On a set of polarized K\"ahler metrics on algebraic manifolds}, J. Diff. Geom. 32 (1990), 99--130

\bibitem{Xi24} M. Xia, {\sl Non-pluripolar products on vector bundles and Chern-Weil formulae}, 
Math. Ann. 390 (2024), 3239--3316

\bibitem{Yau} S.-T. Yau, Nonlinear analysis in geometry, Enseign. Math. 33 (1987), 109--158

\bibitem{Zelditch} S. Zelditch, {\sl Szeg\H{o} kernels and a theorem of Tian}, 
Internat. Math. Res. Notices 6 (1998), 317--331


\end{thebibliography}
\end{document}